\documentclass[a4paper,10pt]{article}
\usepackage{paper-en}
\usepackage{hyperref}

\def\thetitle{An isoperimetric inequality for mean shadow}
\def\theauthors{Roman Karasev, Anton Petrunin, and Alexander Plakhov}

\hypersetup{colorlinks=true,
citecolor=black,
linkcolor=black,
anchorcolor=black,
filecolor=black,
menucolor=black,
urlcolor=black,
pdftitle={\thetitle},
pdfauthor={\theauthors}
}
\begin{document}

\title{\thetitle}
\author{\theauthors}
\date{}
\maketitle

\newcommand{\Addresses}{{\bigskip\footnotesize

\noindent  Roman Karasev,
\par\nopagebreak
\textsc{Institute for Information Transmission Problems RAS, Bolshoy Karetny per. 19, Moscow, Russia 127994.}

\medskip

\noindent   Anton Petrunin,
\par\nopagebreak
 \textsc{Math. Dept. PSU, University Park, PA 16802, USA.}

\medskip

\noindent Alexander Plakhov,
\par\nopagebreak
\textsc{Center for R{\&}D in Mathematics and Applications, Department of Mathematics, University of Aveiro, Portugal.}
\par\nopagebreak
\textsc{Institute for Information Transmission Problems RAS, Bolshoy Karetny per. 19, Moscow, Russia 127994.}

}}

\begin{abstract}
We prove an isoperimetric inequality in terms of the mean shadow area.
\end{abstract}

\section{Introduction}

The $k$-dimensional Hausdorff measure will be denoted by $\mathcal H^k$;
we will use the standard normalization, so the $k$-dimensional unit cube has unit $\mathcal H^k$-measure.

Given a subset $K\subset \RR^n$, let us denote by $\fav_k K$ its $k$-dimensional \emph{mean shadow area}; that is, the average of $\mathcal H^k$-measures of the orthogonal projections of $K$ onto $k$-dimensional subspaces in $\RR^n$.
The mean shadow area is defined for any Borel set; see Proposition~\ref{prop:def}.
The $1$-dimensional mean shadow area $\fav_1K$ is also known as \emph{Favard length}; see \cite{Laba2012,Damian2024}.

\begin{thm}{Theorem}\label{thm:main}
Let $K\subset \RR^n$ be a compact subset and let $B\subset \RR^n$ be a ball.
Then
\[
\fav_{n-1}K= \fav_{n-1}B
\quad\Rightarrow\quad
\mathcal H^n K\le \mathcal H^n B.
\]
Moreover, if equality holds, then $K$ contains a congruent copy of $B$ and the remaining part has vanishing volume.
\end{thm}

\begin{wrapfigure}{r}{43 mm}
\vskip-4mm
\centering
\includegraphics{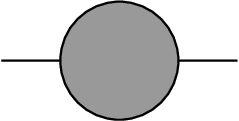}
\end{wrapfigure}

For the equality case, we give only a necessary condition, which is not sufficient; an example can be guessed from the picture.
On the other hand, equality holds for the union of a ball with any set of vanishing mean shadow area, but these are not the only examples; the latter follows from the result of Kenneth Falconer \cite[Theorem 5.1]{falconer}.

By \cite[3.3.13]{federer} we have
\[\frac{n\cdot\omega_n}{\omega_{n-1}}\cdot \fav_{n-1}K
\le\mathcal H^{n-1} (\partial K),\]
where $\omega_n$ denotes the volume of the $n$-dimensional unit ball.
Moreover, by Crofton's formula the equality holds if $K$ is convex.
Therefore, for convex sets our theorem is just a reformulation of the standard isoperimetric inequality,
but in the general case, it is more exact.

\begin{thm}{Corollary}\label{cor:k-shadows}
Let $K\subset \RR^n$ be a compact subset.
Denote by $r_k(K)$ the radius of a ball $B_k$ such that $\fav_k K=\fav_k B_k$.
Then
\[
r_1(K)\ge r_2(K)\ge\dots\ge r_n(K).
\]

Moreover, given a sequence $r_1\ge r_2\ge\dots\ge r_n\ge 0$, there is a compact set $K\subset \RR^n$ with $r_i(K)$ arbitrarily close to $r_i$ for each $i$.
\end{thm}

Note that for convex sets, a stronger statement follows from the Alexandrov--Fenchel inequality, that is, from the logarithmic concavity of the mixed volumes
\[
v_k=V(\underbrace{K,\ldots,K}_k,\underbrace{B,\ldots,B}_{n-k})
\]
as a function of $k$.

\begin{thm}{Corollary}\label{cor:borel}
Let $E\subset \RR^n$ be a Borel subset and let $B\subset \RR^n$ be a ball.
Then
\[
\mathcal H^n E= \mathcal H^n B
\quad\Rightarrow\quad
\fav_{n-1}E\ge \fav_{n-1}B.
\]
\end{thm}

In the case of equality, it is plausible that $E$ coincides with a ball up to measure zero, that is, there is a ball $B$ such that $\mathcal H^n(E\triangle B)=0$, where $\triangle$ denotes symmetric difference.

\medskip

Note that $\fav_k K$ is proportional to the $L^1$-norm of the $\mathcal H^k$-measure of the orthogonal projection of $K$ onto a $k$-dimensional subspace of $\mathbb{R}^n$, viewed as a function of the projection direction.
Let $\fav_k^p K$ denote the $L^p$-norm of the projection measure, $1 < p \le +\infty$.
By Hölder's inequality, $\fav_k^p K$ is greater than or equal to $\fav_k K$ times a positive constant, and the equality can occur only when all projection measures coincide (up to a set of measure zero).
As a result, we obtain the following corollary from Theorem~\ref{thm:main}.

\begin{thm}{Corollary}\label{cor:Lp}
Under the conditions of Corollary~\ref{cor:borel},
\[
\mathcal H^n E= \mathcal H^n B
\quad\Rightarrow\quad
\fav_{n-1}^p E\ge \fav_{n-1}^p B.
\]

\end{thm}

We were motivated by Newton's minimal resistance problem \cite{N}, which can be described as follows.
Consider a body that moves in a rarefied medium of particles, so the mutual interaction of the particles can be neglected.
One is interested in minimizing the time-averaged value of the force of resistance of the medium.

Our isoperimetric problem arises if one assumes that the body moves in a random direction without rotation and that the collisions between the body and the particles are completely inelastic: after colliding with the body, each particle loses its relative velocity and remains forever near the body.
This problem was also posed by Guillaume Aubrun \cite{guillaume}.

On the other hand, if the collisions are perfectly elastic, then one is led to the so-called problem of \emph{camouflaging in billiards}.
Let us briefly discuss this variant, since it is interesting in its own right.
We consider billiard motion in the exterior of some compact set $K$ with sufficiently nice boundary.
For almost every particle that initially
moves along some line $L$ with unit velocity $v$, there is an
asymptotic future velocity $v'$; that is, after finitely many
reflections from $K$, the particle will move forever along a straight
line with unit velocity $v'$.
We define the \emph{visibility index} of the body $K$ as the integral
of this squared change of velocity $|v-v'|^2$ over all lines $L$, with respect to
the natural measure on the space of lines.
We need to find a body of fixed volume with the smallest visibility index.

Among convex bodies, the solution is a ball; in this case the problem reduces to the isoperimetric inequality.
Among connected figures in the plane, the solution is a rough disk; see \cite{ARMA} and \cite[Theorem 6.1]{book2012}.
Compared with the convex case, this reduces the visibility index by approximately $1.22\%$.
Among possibly disconnected bodies or bodies in $\mathbb R^n,\, n \ge 3$, the solution is unknown.
For the class of bodies contained in a fixed ball, lower bounds are known; see \cite{cam1,cam2}.
In the general case, it is not even known whether the infimum is positive.

Additionally, it is known how to construct a body invisible in $m$ directions $(m=1,\,2,\ldots)$, provided that the angles between them are multiples of $\pi/m$; see \cite{invisNpoints}.
Some other results are known about invisibility in several directions or from several points in $\RR^2$ and $\RR^3$ \cite{book2012,fractal,invis2points}.

\paragraph{About the proof.}
\emph{Almgren's frontier} $\cont K$ of a set $K\subset \RR^n$ will be defined
as the intersection of its boundary and the boundary of its convex hull $\conv K$.
We observe that the outer normals to supporting hyperplanes to $\conv K$ at points in $\cont K$ cover the whole sphere.
Further, by monotonicity of mean shadow area with respect to inclusion, we have $\fav_{n-1}(\cont K)\z\le \fav_{n-1}K$.

Now let us assume that $K$ is a counterexample.
Using the two observations above, we find a point $p\in \cont K$ so that chopping $K$ in a small neighborhood of $p$ can produce a worse counterexample.
Then Zorn's lemma leads to a contradiction.

This type of argument was first used by Frederick Almgren in his proof of optimal isoperimetric inequality; see \cite{almgren-1986}.
Our proof might serve as a more elementary illustration of this idea.

\paragraph{Structure of the paper.}
In Section~\ref{sec:borel} we show that mean shadow area is defined for any Borel set and prove Corollary~\ref{cor:borel} modulo the main theorem.
In the following section we prove the main theorem.
In Section~\ref{sec:k-shadows}, we prove Corollary~\ref{cor:k-shadows}.
In the final section we sketch a more elementary proof of the 2-dimensional case;
it is based on the observation that for a connected compact set $K\subset\RR^2$ we have
\[\fav_1K=\fav_1(\conv K)
\quad\text{and}\quad
\mathcal H^2 K\le \mathcal H^2(\conv K).
\]

\section{Borel sets}\label{sec:borel}

\begin{thm}{Proposition}\label{prop:def}
The mean shadow area $\fav_kE$ is defined for any Borel set $E\subset \RR^n$ and any integer $1\le k\le n$.
\end{thm}

\begin{proof}
Consider the Borel set
\[
\widetilde E = \{(x,R)\in \RR^n\times \mathrm{O}(n)\ |\ R(x)\in E\}.
\]
The mean shadow area $\fav_kE$ can be defined as the measure of the projection of $\widetilde E$ to $\RR^k\times \mathrm{O}(n)$, assuming the normalized Haar measure on $\mathrm{O}(n)$.
By Suslin's theorem \cite[Section~1.10]{bogachev-i}, a projection of a Borel set is Lebesgue measurable, so $\fav_kE$ is defined.
\end{proof}

\begin{proof}[Proof of Corollary~\ref{cor:borel}]
From the inner regularity of the Lebesgue measure, any Borel set $E$ can be approximated in measure by compact subsets $K_i \subset E$, that is, $\mathcal H^n K_i \to \mathcal H^n E$ as $i \to \infty$.
From the inclusion $K_i\subset E$, we have $\fav_{n-1}E \z\ge \fav_{n-1}K_i$.
Hence the inequality for $E$ follows from the inequality for every $K_i$.
It remains to apply the main theorem.
\end{proof}

\section{Compact sets}\label{sec:compact}

Given a compact set $K$, consider the ball $B$ such that $\fav_{n-1}K=\fav_{n-1}B$.
Let us denote by $\kappa=\kappa(K)$ the normal curvature of $\partial B$, so $\tfrac1\kappa$ is the radius of $B$ and
\[
\frac{\omega_{n-1}}{\kappa^{n-1}}=\fav_{n-1}K=\fav_{n-1}(B),\quad \frac{\omega_{n}}{\kappa^{n}}=\mathcal H^n B.
\]

\begin{thm}{Lemma}\label{lem:support}
Let $K$ be a compact set in $\RR^n$.
Given $\eps>0$, there is a strictly convex body $C$ with smooth boundary $\partial C$ such that
\begin{enumerate}[(i)]
 \item $C\supset K$,
 \item $\partial C\cap K$ contains a single point, say $p$, and
 \item the mean curvature of $\partial C$ at $p$ is at least $(n-1)\cdot (\kappa(K)\z-\eps)$.
\end{enumerate}
Moreover, the same holds for some $\eps<0$ unless $\partial K\supset \partial(\conv K)$ or $\mathcal H^n K=0$.
\end{thm}

The second part of the lemma (with $\eps<0$) will be needed in the proof of the equality case.
One can prove that it holds if $K$ is not a ball, but we state only the weaker version needed in our proof.

Let us define \emph{Almgren's frontier} of a subset $K\subset \RR^n$ as
\[\cont K=\partial K\cap\partial(\conv K).\]

\parit{Proof for $\eps>0$.}
Let us denote by $K_s$ the closed $s$-neighborhood of $K$.
If $s>0$, then $\conv K_s$ has $C^{1,1}$-smooth boundary; in particular, its second derivatives exist almost everywhere and are bounded.
So the shape operator ($\Shape$), the Gauss curvature $G=\det(\Shape)$, and mean curvature $H=\tr(\Shape)$ are defined almost everywhere on $\partial(\conv K_s)$.
Note also that any hyperplane projection of $K_s$ is the closed $s$-neighborhood of the projection of $K$ to the same hyperplane, which implies continuous dependence of the area of the projection on $s\ge 0$.

Since every linear function attains its maximum on $\conv K_s$ at a point of $K_s$, the normals of $\cont K_s$ cover the whole sphere.
Hence,
\[
\int\limits_{x\in \cont K_s}G(x)=\mathcal H^{n-1} \SSS^{n-1}=n\cdot\omega_n.
\]
Therefore, the inequality
\[
G(x)\ge \frac{n\cdot\omega_n}{\mathcal H^{n-1} (\cont K_s)}
\]
holds for $x$ in a set of positive measure of $\cont K_s$.

From the inclusion $\cont K_s\subset K_s$ we have
\[
\fav_{n-1}(\cont K_s)\le \fav_{n-1}K_s.
\]
Further, since $\cont K_s\subset \partial (\conv K_s)$, the Crofton formula implies
\[
\mathcal H^{n-1} (\cont K_s)\le \frac{n\cdot\omega_n}{\omega_{n-1}}\cdot \fav_{n-1}(\cont K_s);
\]
the constant $\tfrac{n\cdot\omega_n}{\omega_{n-1}}$ is optimal; we have equality if $\cont K_s$ is a closed convex hypersurface.

Combining the last three inequalities with the definition of $\kappa(K_s)$ and its continuity in $s$ we get a point $x\in \cont K_s$ with well-defined shape operator and Gauss curvature at least $\kappa(K_s)^{n-1}$.

Let us take a smooth hypersurface $W_s$ with slightly smaller (with positive definite difference) shape operator at $x$ compared to that of $\partial(\conv K_s)$;
note that $W_s$ touches $K_s$ only at $x$.
Note that all the curvatures of $\partial \conv K_s$ were at most $\tfrac1s$, so all the curvatures of $W_s$ may be assumed strictly smaller than $\tfrac1s$.

\begin{wrapfigure}{r}{49 mm}
\vskip-0mm
\centering
\includegraphics{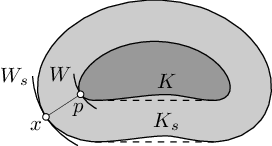}
\end{wrapfigure}

Let $W$ be the inner $s$-equidistant hypersurface of $W_s$.
Note that $W$ touches $K$ in a single point, say $p$, which is the nearest-point projection of $x$ to $K$.
From the above observation on curvatures it follows that $W$ is still convex and smooth at $p$.
It remains to take a small smooth closed strictly convex neighborhood $N$ of $\conv K$, chop it by $W$ and smooth the resulting body near the corner $W\cap \partial N$ while keeping strict convexity.
The hypersurface $W$ was already strictly convex, so $C$ is strictly convex in a neighborhood of $p$.
It may be made strictly convex by smoothing its support function outside of the corresponding neighborhood of the normal at $p$.

The bound on the mean curvature follows from the inequality of arithmetic and geometric means.
\qeds

\begin{proof}[Proof for $\eps<0$]
Now assume $\mathcal H^n K\ne 0$; so, $\conv K$ is nondegenerate.
If $\partial K\z{\not\supset} \partial(\conv K)$, then the Crofton formula implies
\[
\mathcal H^{n-1} (\cont K) < \frac{n\cdot\omega_n}{\omega_{n-1}}\cdot \fav_{n-1}(\cont K).
\]
Note that $\cont K_s$ is the set of points $x\in \partial(\conv K_s)=\partial (\conv K)_s$ whose closest point $p$ in $\conv K$ belongs to $\cont K$.

Let us compare $\mathcal H^{n-1}(\cont K_s)$ and $\mathcal H^{n-1}(\cont K)$ using the Crofton formula;
namely, if $X$ lies in the boundary of a convex set, then $\mathcal H^{n-1} X$ is the integral of the number of points in $X\cap \ell$ over straight lines $\ell$.
Almost every line $\ell$ either does not intersect $\partial(\conv K)$ or intersects it in two smooth points of $\partial(\conv K)$.
Consider only such lines.
We have the following cases:

Case 1: $\ell$ has no intersection with $\conv K$.
Then $\ell$ does not intersect $\conv K_s$ either for sufficiently small $s$.

Case 2: both points of $\ell\cap\partial(\conv K)$ do not belong to $\cont K$.
Then they both are at positive distance from $\cont K$ and $\ell$ does not intersect $\cont K_s$ for sufficiently small $s$.
Indeed, the two points of intersection $\ell\cap\partial(\conv K_s)$ tend to the two points of intersection $\ell\cap \partial(\conv K)$ as $s\to+0$, while $\cont K_s$ is a distance at most $s$ from $\cont K$.

Case 3: one of the points of $\ell\cap\partial(\conv K)$ does not belong to $\cont K$.
Then this point is at positive distance from $\cont K$.
Therefore, the corresponding point of $\ell\cap \partial(\conv K_s)$ does not belong to $\cont K_s$ for sufficiently small $s$ as in the previous case.

Case 4: both points of $\ell\cap\partial(\conv K)$ are in $\cont K$.
This does not affect the conclusion below.

In all the cases, as $s\to +0$, the integrand of the Crofton formula for $\mathcal H^{n-1}(\cont K_s)$ almost everywhere has upper limit less than or equal to the integrand of the Crofton formula for $\mathcal H^{n-1}(\cont K)$.
By Fatou's lemma,
\[\varlimsup_{s\to +0} \mathcal H^{n-1}(\cont K_s) \le \mathcal H^{n-1}(\cont K).\]



It follows that
\[
\mathcal H^{n-1} (\cont K_s)\le \frac{n\cdot\omega_n}{\omega_{n-1}}\cdot \fav_{n-1}(\cont K)-\delta,
\]
for some fixed $\delta>0$ and any small $s>0$.
Now the argument above with $\fav_{n-1}(\cont K_s)$ replaced with $\fav_{n-1}(\cont K)\le \fav_{n-1}K$ implies the statement for some $\eps<0$ depending on $\delta$.
\end{proof}

\begin{thm}{Lemma}\label{lem:variation}
Let $C$, $K$, and $p$ be as in the previous lemma, and let
\[
K_t=\set{x\in K}{\dist(\partial C, x)\ge t}.
\]
Then
\[
(1+o_t(1))\cdot(\mathcal H^n K-\mathcal H^n K_t)\le \frac{n\cdot \omega_n}{\omega_{n-1}\cdot H(p)}\cdot(\fav_{n-1}K-\fav_{n-1}K_t),
\]
where $H(p)$ denotes the mean curvature of $\partial C$ at $p$.
\end{thm}

\begin{proof}
Let $u\in \SSS^{n-1}$ and $r\in \RR$.
Consider the unique point $\zeta(u)\in \partial C$ with outer normal vector $u$.
Then the map
\[
\xi\:(u,r)\mapsto\zeta(u)-r\cdot u
\]
gives a smooth parametrization of a neighborhood of $\partial C$.
Denote by $u_0$ the unit normal vector at $p$, so $p=\xi(u_0,0)$.

Denote by $\mathbf 1_K$ the characteristic function of $K$ and let $\phi(u,r)=\mathbf 1_K\circ \xi(u,r)$.
Then
\begin{align*}
\mathcal H^n K-\mathcal H^n K_t
=
&\int\limits_{u\in \SSS^{n-1}}
\int\limits_{r\in [0,t]} \phi(u,r)\cdot J(u,r)=
\\
=
\frac{n\cdot \omega_n\cdot (1+o_t(1))}{G(p)}
\cdot
&\oint\limits_{u\in \SSS^{n-1}}
\int\limits_{r\in [0,t]} \phi(u,r),
\end{align*}
where $\oint$ denotes the average along the sphere, and $J(u,r)$ is the Jacobian of $\xi$ at $(u,r)$.
The last equality follows since $J(u_0,0)=1/G(p)$ and if $r\le t$ for small $t>0$, then $\phi(u,r)\ne 0$ only if $u$ is in a sufficiently small neighborhood of $u_0$, so we may approximate the Jacobian by its value at $(u_0,0)=\xi^{-1}(p)$.

Let us turn to hyperplane projections.
Choose a unit vector $w$; denote by $L$ and $L_t$ the projections of $K$ and $K_t$ to $w^\perp$.
The $t$-neighborhood of $L_t$ is contained in the projection of $C$, so $L_t$ is at distance at least $t$ from the boundary of the projection of $C$.
This allows us to estimate
\[
\mathcal H^{n-1} L-\mathcal H^{n-1} L_t
\ge
\int\limits_{u\in \SSS^{n-2}}
\int\limits_{r\in [0,t]}
\phi(u,r)\cdot J_w(u,r),
\]
where $\SSS^{n-2}=\SSS^{n-1}\cap w^\perp$ and
$J_w(u,r)$ denotes the Jacobian at $(u,r)$ of the projection onto $w^\perp$ of the restriction $\xi|_{\SSS^{n-2}\times \RR}$.
This estimate holds since the points with $\phi(u,r)=1$ are projected injectively to $L$ with the Jacobian $J_w$ and do not get into $L_t$.

If $t$ is sufficiently small, then $\phi(u,r)\ne 0$ only if $u$ lies in a sufficiently small neighborhood of $u_0$.
This implies that $w$ may be considered almost perpendicular to $u_0$, which we assume further.

Denote by $w'$ the vector closest to $w$ in $\SSS^{n-1}\cap u_0^\perp$.
If $\phi(u,r)\ne 0$ and $0\le r\le t$, then, from the uniform continuity of $J_w$, we have
\[
J_w(u,r)=(1+o_t(1))\cdot J_{w'}(u_0,0)=(1+o_t(1))\cdot \frac{k(w')}{G(p)},
\]
where $k(w')$ denotes normal curvature of $\partial C$ in the direction $w'$ at $p$.
The last equality can be understood by approximating $C$ by the graph of a quadratic function, noting that a projection corresponds to the restriction of its dual, and noting that a diagonal element of an inverse matrix (corresponding to $k(w')$) is an $(n-2)\times (n-2)$ minor of the original matrix (corresponding to the Gaussian curvature of the projection) divided by the determinant of the whole matrix (corresponding to the Gaussian curvature before the projection).

Note that $\tfrac{1}{n-1}\cdot H(p)$ is the average of $k(w')$ for $w'\perp u_0$.
Therefore, taking the average, we get
\[
\fav_{n-1}K-\fav_{n-1}K_t\ge \frac {H(p)\cdot \omega_{n-1}\cdot(1+o_t(1))}{ G(p)}
\cdot\oint\limits_{u\in \SSS^{n-1}}
\int\limits_{r\in [0,t]}
\phi(u,r).
\]
Here we used that averaging over the whole sphere is the same as averaging over a variable point in a subsphere and then averaging over subspheres.
\end{proof}

\begin{proof}[Proof of Theorem~\ref{thm:main}.]
Assume that $K$ is a counterexample.
Then we can choose $\delta>0$ such that
\[
s_\delta(K) \df \mathcal H^n K - (1+\delta)\cdot \mathcal H^n B > 0,
\]
where $B$ denotes the ball such that $\fav_{n-1}K=\fav_{n-1}B$.

Remove from $K$ the set of measure zero of its points of zero density and then pass to the closure.
After this procedure, $K$ remains compact, its volume does not change, while mean shadow area $\fav_{n-1}K$ may decrease.
Moreover, now every nonempty intersection of $K$ with an open set has positive volume.
This ensures that the volume cut from $K$ by the procedure of Lemma~\ref{lem:variation} will be positive.

Applying Lemma~\ref{lem:variation} to $C$ and $p$ provided by Lemma~\ref{lem:support}, we get a compact set $K'=K_t\subset K$ for small $t>0$ such that
\[
s_\delta(K')>s_\delta(K).\leqno({*})
\]
Indeed, since $\mathcal H^n B=\tfrac{\omega_{n}}{\kappa(K)^{n}}=\tfrac{\omega_n\cdot \fav_{n-1}(K)^{n/(n-1)}}{\omega_{n-1}^{n/(n-1)}}$, we have
\begin{align*}
\mathcal H^n K-\mathcal H^n K'&< \left( \frac{n\cdot \omega_n}{(n-1)\cdot\omega_{n-1}\cdot \kappa(K)}+\eps\right)\cdot (\fav_{n-1}K-\fav_{n-1}K')=
\\
&=\left( \frac{n\cdot \omega_n\cdot \fav_{n-1}(K)^{1/(n-1)}}{(n-1)\omega_{n-1}^{n/(n-1)}}+\eps\right)\cdot (\fav_{n-1}K-\fav_{n-1}K').
\end{align*}
Taking the derivatives of $s_\delta(K)$ as a function of $\mathcal H^n K$ and $\fav_{n-1}K$, we get
\[\frac{\partial s_\delta(K)}{\partial\mathcal H^nK} = 1,
\quad
\frac{\partial s_\delta(K)}{\partial \fav_{n-1} K} = -(1+\delta)\cdot \frac{n\cdot\omega_n}{(n-1)\cdot\omega_{n-1}^{n/(n-1)}} \cdot \fav_{n-1}(K)^{1/(n-1)}.\]
Therefore, we get $({*})$ for an appropriate choice of $\eps>0$.

Now consider an inclusion chain $\mathcal Z$ of compact subsets $L\subset K$ satisfying $s_\delta(L)\z\ge s_\delta(K)$ for all $L\in\mathcal Z$.
The intersection of the chain is a compact set, say~$M$.
Let us show that
\[s_\delta(M)\ge s_\delta(K).\leqno({*}{*})\]
Indeed, choose an open $U\supset M$ such that $\mathcal H^n U\approx\mathcal H^n M$.
Since the sets in the chain are compact, some $L\in\mathcal Z$ lies in $U$.
Hence $\mathcal H^n L\le\mathcal H^n U$;
thus we can choose $L\in\mathcal Z$ such that $\mathcal H^n L$ is arbitrarily close to $\mathcal H^n M$.
Since $L\supset M$, we also have $\fav_{n-1}L\ge \fav_{n-1}M$, and hence $({*}{*})$.

Applying Zorn's lemma to the compact subsets $L\subset K$ satisfying $s_\delta(L)\z\ge s_\delta(K)$, we obtain an inclusion-minimal one, say $M$.
But the above argument applied to $M$ in place of $K$ contradicts the minimality.

Now assume that equality holds for $K$.
After subtracting a set of volume zero we can assume that nonempty intersections of $K$ with open sets have positive volume.
After that we need to show that $K$ is congruent to $B$.
If $\partial K\z\supset \partial(\conv K)$, then it follows from the equality case in the standard isoperimetric inequality.
Otherwise, repeating the above argument with $\delta=0$ instead of $\delta>0$ and taking into account the second part of \ref{lem:support}, we produce a compact subset $K'\subset K$ that contradicts the already proved inequality and arrive at a contradiction.
\end{proof}

\section{Mean \textit{k}-dimensional shadows}\label{sec:k-shadows}

\begin{proof}[Proof of Corollary~\ref{cor:k-shadows}.]
To prove the inequalities
\[
r_1(K)\ge r_2(K)\ge\dots\ge r_n(K),
\]
consider $k$-dimensional projection as a sequence of codimension $1$ projections and averaging over them all.

Let us approximate an $r$-ball by a polyhedron $P$ and then consider the $k$-skeleton $P^{(k)}$ of $P$.
Note that $r_i(P^{(k)})=0$ for $i>k$
and $r_i(P^{(k)})\approx r$ for $i\le k$.
Thus we approximate any sequence of the form
\[(\underbrace{r,\ldots,r}_k,\underbrace{0,\ldots,0}_{n-k}).\]
An arbitrary decreasing sequence can be approximated by a union of such $k$-dimensional sets for $k=1,\ldots,n$ placed sufficiently far apart from each other to make the average projection almost additive.
This proves the second part.
\end{proof}

\section{2-dimensional case}\label{sec:2D}

Here we outline a more elementary proof of Theorem~\ref{thm:main} in the 2-dimensional case.
Note that $\fav_1(K)$ is the average projection length; we will call it the Favard length.

\begin{proof}[Sketch of 2-dimensional case of \ref{thm:main}]
Suppose that $K\subset \RR^2$ is connected.
Since each projection of $K$ coincides with the corresponding projection of its convex hull $\conv K$, we have
\[\fav_1K = \fav_1(\conv K).\]
Since $\mathcal H^2K \le \mathcal H^2(\conv K)$, the problem reduces to the classical isoperimetric problem: minimize the perimeter of a convex set with fixed area.
In this case, the equality holds only for round disks.

Now consider the case when $K\subset \RR^2$ has a finite number of connected components, say $N$.
Repeating the argument above, one can assume that each connected component of $K$ is convex.

Let us show that in the class of sets with fixed area and no more than $N$ connected components, there is a set that minimizes the Favard length.

Consider a sequence of sets $K_k$ with prescribed area and with $\fav_1(K_k)$ tending to the lower bound.
Let $K_k^1, \ldots, K_k^N$ be the convex connected components of $K_k$; some of them may be empty.
Note that the perimeters and diameters of all $K_k^m$ are uniformly bounded, since $\fav_1K_k^m$ is uniformly bounded.
Choose a center $c_k^m\in K_k^m$.
By the Blaschke selection theorem~\cite[Section~6.1]{gruber2007}
, after passing to a subsequence the convex sets $K_k^m-c_k^m$ will tend to some convex set $K_\infty^m$ as $k\to\infty$ in the Hausdorff metric.
Hence we may replace each set $K_k^m$ with the translate $K_\infty^m + c_k^m$ keeping the same total area and keeping the limit of $\fav_1(K_k)$.
Indeed, the difference between the projection lengths (of unions of $N$ segments) is bounded by $2\cdot N$ times the Hausdorff distance.

Consider $K_k$ consisting of different translates $K_k^m = K_\infty^m+c_k^m$ of the fixed collection of $N$ convex sets.
For a pair of indices $i\neq j$, consider the sequence of distances $\dist(K_k^i, K_k^j)$. Either the sequence tends to infinity, or passing to a subsequence we make it bounded.
Doing so for every pair $i\neq j$, we may consider a graph $G$ on vertices $\{1,2,\ldots,N\}$ consisting of pairs $\{i,j\}$ with bounded $\dist(K_k^i, K_k^j)$ and having $\dist(K_k^i, K_k^j)\to +\infty$ for $\{i,j\}\z\notin G$.
If $G$ is connected then the sets $K_k$ have uniformly bounded diameter, so after a translation we may assume $c_k^m\to c_\infty^m$ by compactness and have an actual Hausdorff limit $K_k\to K$, which is the minimizer we need.

If $G$ is not connected, then splitting it into connected components corresponds to nontrivial splitting of $K_k$ into uniformly bounded sets $K_k^{(1)}, \ldots, K_k^{(c)}$, $c\ge 2$.
By the same compactness argument, after passing to a subsequence we may assume that $K_k^{(\ell)} - c_k^{(\ell)}$ tends to some $K_\infty^{(\ell)}$.
Then we replace $K_k^{(\ell)}$ with $K_\infty^{(\ell)} \z+ c_k^{(\ell)}$ keeping the area and the limit $\lim\limits_{k\to\infty}\fav_1(K_k)$.
After these manipulations every distance $\dist(K_k^{(\ell)}, K_k^{(\ell')})$ tends to infinity when $\ell\neq\ell'$.
This means that $\fav_1(K_k)$ tends to the sum $\sum_{\ell} \fav_1(K_\infty^{(\ell)})$.
But placing the same sets $K_\infty^{(\ell)}$ in arbitrary fixed position makes the average projection length of their union strictly smaller than $\sum_{\ell} \fav_1(K_\infty^{(\ell)})$ because of non-trivial overlap of their projections.
Hence $\sum_{\ell} \fav_1(K_\infty^{(\ell)})$ cannot be equal to $\inf \fav_1K$ over unions of $N$ convex sets and the case of non-connected $G$ is impossible.

Applying a variational argument similar to our main proof, we see that $\conv K$ has $C^1$-smooth boundary and $\cont K$ is a finite collection of circular arcs of radius
\[
R=\frac{2\cdot \mathcal H^2 K}{\fav_1 K}.
\leqno({*})
\]

Let us show that
\[\fav_1 K \ge 2\cdot \pi\cdot R,\leqno({*}{*})\]
and the equality is attained only if $K$ is a disk.
Indeed, $\fav_1 K\ge \fav_1(\cont K)$; the set $\cont K$ is the union of circular arcs of radius $R$ with total angular measure $2\cdot\pi$.
By the Crofton formula $\cont K$ contributes at least $2\cdot\pi\cdot R$ to $\fav_1 K$, but if $\partial(\conv K)\setminus \partial K\ne \emptyset$, this contribution is strictly larger than $2\cdot\pi\cdot R$.

Now observe that $({*})$ and $({*}{*})$ imply the statement in the theorem.

Finally, assume $K$ is a counterexample.
Consider its $\eps$-neighborhood $U_\eps$ with $\eps > 0$.
We have $\mathcal H^2 (U_\eps) \to \mathcal H^2 K$ as $\eps \to 0$.
Further, $\fav_1(U_\eps) \to \fav_1 K$ as $\eps \to 0$.
It follows that the infimum in the problem coincides with the infimum in the class of open bounded sets.

Each bounded open set $K$ has at most countably many connected components, say $C_1, C_2,\ldots$, and therefore can be approximated by open sets
\[K_i =C_1\cup\ldots\cup C_i.\]
We have $\mathcal H^2 K_i  \to \mathcal H^2K$ as $i \to \infty$ and $\fav_1K_i \le \fav_1K$ always.
It follows that the infimum in the problem coincides with the infimum in the class of open bounded sets with finitely many connected components, which is impossible.
\end{proof}

\parbf{Acknowledgments.}
We were partially supported by the following grants:
Roman Karasev by the state assignment 1.1.1-0029/25 of Ministry of Science and Higher Education of the Russian Federation for IITP RAS;
Anton Petrunin by the National Science Foundation grant DMS-2005279;
Alexander Plakhov by CIDMA through the Portuguese Foundation for Science and Technology (FCT), grants UID/4106/2025 and UID/PRR/4106/2025.

{\sloppy
\def\emph{\textit}
\printbibliography[heading=bibintoc]
\fussy
}

\Addresses

\end{document}